%% file: arccat.tex
\documentclass[10pt]{amsart}

\usepackage{amsmath,amssymb,amsthm,amsfonts,verbatim,comment}

\usepackage{enumerate}
\usepackage{color}
\usepackage{graphicx}

\usepackage{tikz-cd}

\usepackage[
pdfborderstyle={},
pdfborder={0 0 0},
pagebackref,
pdftex]{hyperref}

    \newtheorem{thm}{Theorem}[section]
    \newtheorem{prop}[thm]{Proposition}
    \newtheorem{lem}[thm]{Lemma}
    \newtheorem{cor}[thm]{Corollary}

    \theoremstyle{definition}

    \newtheorem{defn}[thm]{Definition}

    \theoremstyle{remark}
    \newtheorem{rem}[thm]{Remark}

\newcommand{\coloneq}{\mathrel{\mathop:}\mkern-1.2mu=}

\newcommand{\executeiffilenewer}[3]{%
\ifnum\pdfstrcmp{\pdffilemoddate{#1}}%
{\pdffilemoddate{#2}}>0%
{\immediate\write18{#3}}\fi%
}

\newcounter{richardcomments}

\newcommand{\fakeenv}{}

\newenvironment{restate}[2]                                    
{ 
 \renewcommand{\fakeenv}{#2}                              
 \theoremstyle{plain} 
 \newtheorem*{\fakeenv}{#1~\ref{#2}}                
 \begin{\fakeenv}
}
{
 \end{\fakeenv}
}

\title{Contractible, hyperbolic but non-CAT(0) complexes}

\author{Richard~C.~H.~Webb}
 \email{richard.webb@manchester.ac.uk}

\begin{document}

\begin{abstract}

We prove that almost all arc complexes do not admit a CAT(0) metric with finitely many shapes, in particular any finite-index subgroup of the mapping class group does not preserve such a metric on the arc complex. We also show the analogous statement for all but finitely many disc complexes of handlebodies and free splitting complexes of free groups. The obstruction is combinatorial. These complexes are all hyperbolic and contractible but despite this we show that they satisfy no combinatorial isoperimetric inequality: for any $n$ there is a loop of length $4$ that only bounds discs consisting of at least $n$ triangles.

On the other hand we show that the curve complexes satisfy a linear combinatorial isoperimetric inequality, which answers a question of Andrew~Putman.

\end{abstract}

\maketitle

\section{Introduction}

In general the mapping class group $\mathrm{Mod}(S)$ cannot act properly by semisimple isometries on a complete CAT(0) space   \cite{KapovichLeeb, BridsonHaefliger, Bridson}, in particular, it is not a CAT(0) group. However, the Teichm\"uller space with the Weil--Petersson metric is CAT(0) \cite{Tromba,WolpertNegative,WolpertGeodesic}, furthermore so is its completion \cite[Corollary~II.3.11]{BridsonHaefliger}, and the mapping class group acts on this by semisimple isometries \cite{DaskalopoulosWentworth}. When $g\geq 3$, Bridson \cite{Bridson} used this action to show that any non-trivial homomorphism \[\mathrm{Mod}(S_{g,p})\to\mathrm{Mod}(\Sigma),\]must send Dehn twists to roots of multitwists. This result was then used by Aramayona--Souto \cite{AramayonaSouto} to classify---under topological assumptions on $S_{g,p}$ and $\Sigma$---all non-trivial homomorphisms from $\mathrm{Mod}(S)$ to $\mathrm{Mod}(\Sigma)$. Despite the fact that the mapping class group is not CAT(0) in general, the study of its algebra has been enhanced by its action on CAT(0) spaces.

Perhaps the most striking application of actions on CAT(0) spaces has been provided by CAT(0) cube complexes and their role in  the proof of the virtual Haken conjecture, see \cite{Agol}. Not only did this uncover a difficult topological consequence but there were many exciting \textit{algebraic} consequences for the fundamental groups of closed, hyperbolic $3$-manifolds such as largeness, LERF, linearity over $\mathbb{Z}$, bi-orderability, conjugacy separability, see for example \cite{AFW} and the references therein.

One theme of this paper concerns the problem of finding CAT(0) metrics on a given space. The spaces of interest in this paper are locally infinite complexes, such as the arc complex of a surface and the free splitting complex of a free group. We say that a complex $K$ (equipped with a metric) has \textit{finitely many shapes} if there are only finitely many isometry classes of simplex in $K$. We show the following

\begin{thm}\label{thm:fin}

Whenever $K$ is a (not necessarily locally compact) flag simplicial complex equipped with a CAT(0) metric with finitely many shapes then $K$ satisfies a quadratic combinatorial isoperimetric inequality.

\end{thm}

Any CAT(0) space satisfies a quadratic \textit{coarse} isoperimetric inequality. We remark that this observation does not suffice to prove Theorem~\ref{thm:fin}. We point the reader to Figure~\ref{complex} for an elementary example.

As far as the author is aware, Theorem~\ref{thm:fin} is not in the literature and is new. Perhaps this is because the focus of the isoperimetric inequality has mainly been on finitely presented groups, where the optimal combinatorial isoperimetric inequality and the optimal coarse isoperimetric inequality (the \textit{Dehn function}) are equivalent for Cayley complexes. Cayley complexes are always locally compact with a cocompact group action---both of these properties are necessary for the equivalence between the coarse and the combinatorial.

In Section~\ref{nocomb} we use Theorem~\ref{thm:fin} to show that the majority of arc complexes do not admit a CAT(0) metric with finitely many shapes. The same holds for all but finitely many disc complexes of handlebodies and all but finitely many free splitting complexes.

\begin{thm}\label{thm:bigbig}

Let $K$ be one of the following complexes.
\begin{itemize}
\item The arc complex $\mathcal{A}(S_{g,p})$ where $g\geq 2$ and $p\geq 2$ (or $p\geq 6-2g$ when $g=0$ or $1$).
\item The disc complex $\mathcal{D}_n$ of a handlebody of genus $n\geq 5$.
\item The free splitting complex $\mathcal{FS}_n$ of a free group of rank $n\geq 5$.
\end{itemize}
Then there is a family of loops $c_N$ of combinatorial length $4$ in $K^{(1)}$ such that the following holds. Whenever $P$ is a triangulation of a surface with one boundary component and $f\colon P \to K^{(2)}$ is a simplicial map where $f|_{\partial P}$ maps bijectively onto $c_N$ then $P$ must have at least $N$ triangles.
In particular $K$ does not admit a CAT(0) metric with finitely many shapes.
\end{thm}

In other words, each of the above complexes does not satisfy any combinatorial isoperimetric inequality at all. Note that all of the above complexes are contractible \cite{Harer85,Hatcher91,McCullough,Hatcher95}, hyperbolic \cite{MasurSchleimer,HandelMosher}, and such that any finite group action must fix some point \cite{Kerckhoff,Harer,HOP}, and so the usual and more well-known obstructions to being CAT(0) do not apply. Our obstruction to being CAT(0) with finitely many shapes is new.

On the other hand the above complexes are in stark contrast to the curve complexes and the arc-and-curve complexes, as we show

\begin{restate}{Theorem}{linearcombi}

The curve complex $\mathcal{C}(S)$ and arc-and-curve complex $\mathcal{AC}(S)$ satisfy a linear combinatorial isoperimetric inequality.
\end{restate}

In \cite[p.~104]{MasurMinskyI} Masur~and~Minsky state that it is an interesting question whether the arc-and-curve complexes $\mathcal{AC}(S)$ admit a CAT($\kappa$) metric for some $\kappa\leq 0$. Note that any example is immediately also CAT(0). A bold guess based on Theorem~\ref{thm:bigbig} seems to be no. The proof of Theorem~\ref{thm:fin} also applies to higher-dimensional combinatorial isoperimetric inequalities, see Theorem~\ref{higherthm}, and at present it is unclear whether $\mathcal{AC}(S)$ should satisfy any at all. This avenue might lead to a negative answer to Masur~and~Minsky's question. We note that this question is important: if the answer is yes with an isometric action of a finite-index subgroup of $\mathrm{Mod}(S)$, then it might prove to be a useful tool for understanding $\mathrm{Mod}(S)$ and its quotients in a similar vein to \cite{DahmaniHS}.

We sketch the idea of the proof of Theorem~\ref{linearcombi} at the start of Section~\ref{linearsec}. The most important feature is the use of Masur~and~Minsky's tightening procedure i.e. the idea of the proof of the existence of tight geodesics \cite[Lemma~4.5]{MasurMinskyII}.

Presently the tightening procedure (and tight geodesics) is one of the missing pieces in the analogous Masur--Minsky theory for $Out(F_n)$. Such a theory could prove useful for example in showing that the action on the free factor complex is acylindrical, which seems to be open, but known for the curve complex \cite{BowditchTight}. However much intense progress has been made in discovering/creating analogues of hyperbolicity \cite{BestvinaFeighnHyp, HandelMosher,HilionHorbez} and subsurface projection \cite{BestvinaFeighnSubfactor,Taylorproj}. Theorem~\ref{linearcombi} is a new phenomenon that is not yet observed in the $Out(F_n)$ theory. Because the tightening procedure leads to a proof of Theorem~\ref{linearcombi}, there is an implied relationship between the two. The following question is therefore interesting: is there a cocompact complex for $Out(F_n)$, analogous to $\mathcal{C}(S)$, that satisfies a linear combinatorial isoperimetric inequality?

We remark that it is currently not known whether the pants complex satisfies a quadratic combinatorial isoperimetric inequality. In light of Theorem~\ref{thm:fin} this is related to a question of Brock who asked whether there is a CAT(0) complex analogous to the pants complex (it is already a theorem of Brock that the pants complex is quasi-isometric to the Weil--Petersson metric \cite{Brock}). Recently Islambouli~and~Klug showed that any loop in the $1$-skeleton of the pants complex of a closed surface determines a smooth $4$-manifold \cite{IslambouliKlug}, and that any smooth $4$-manifold arises by some such loop. They also define a signature that provides a lower bound on the number of triangles for any disc bounding that loop. It would seem that the structure of smooth $4$-manifolds could give new insight on the pants complex, or possibly the other way round, or both.

\subsection*{Acknowledgements}

The author is indebted to Andrew~Putman, who asked the question whether the curve complex satisfies a linear combinatorial isoperimetric inequality, and who pointed out that this does not immediately follow from hyperbolicity. The other projects pursued in this paper naturally stemmed from this question.

The author wishes to thank Mladen~Bestvina, Martin~Bridson, Daniel~Groves, Radhika~Gupta, Sebastian~Hensel, Piotr~Przytycki, Andrew~Putman, and Henry~Wilton for useful comments and conversations. We thank Mark~Bell who pointed out that the proof of Theorem~\ref{bigthm} could be promoted from discs to all surfaces with one boundary component.

This work was supported earlier by the Stokes Research Fellowship of Pembroke College, University of Cambridge, and currently the EPSRC Fellowship EP/N019644/2.

\section{Combinatorial isoperimetric inequalities} \label{sec:quad}

In this section we show the following, which is a generalisation of Theorem~\ref{thm:fin}. The notions of \textit{bounded shapes} and \textit{thick shapes} are defined below.

\begin{thm} \label{thmcat}
Whenever $K$ is a (not necessarily locally compact) flag simplicial complex equipped with a CAT(0) metric with bounded, thick shapes then $K$ satisfies a quadratic combinatorial isoperimetric inequality.
\end{thm}

It is well known that CAT(0) spaces satisfy a quadratic \textit{coarse} isoperimetric inequality (see \cite[Chapter~III.H.2.4]{BridsonHaefliger}) but for complexes that are not locally compact this does not suffice to prove a quadratic \textit{combinatorial} isoperimetric inequality, see Figure~\ref{complex} for an example. In fact any two of the three hypotheses (CAT(0), bounded shapes, or thick shapes) of Theorem~\ref{thmcat} are satisfied by some metric on the complex given in Figure~\ref{complex} and so all three hypotheses are necessary.

The main theorems of this paper concern simplicial complexes and therefore for brevity we focus on this case. We expect that Theorem~\ref{thmcat} is true for more general complexes $K$ by using the methods given here but we have not checked the details.

\begin{figure}
\begin{center}
\executeiffilenewer{complex.svg}{complex.pdf}%
{inkscape -z -D --file=complex.svg %
--export-pdf=complex.pdf --export-latex}%
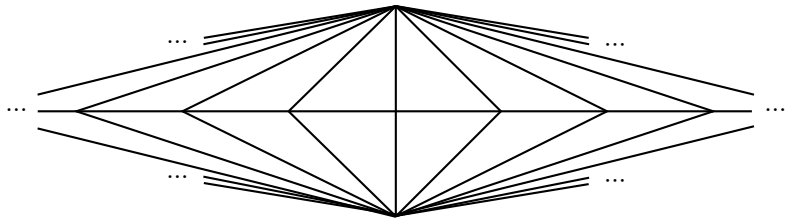%

\caption{A natural triangulation of the suspension of $\mathbb{R}$ produces a non-locally compact, 2-dimensional simplicial complex that satisfies no combinatorial isoperimetric inequality: there are loops of combinatorial length four that require arbitrarily many triangles to deform them to a point. From an appropriate embedding in the euclidean plane it can be endowed with a complete CAT(0) metric.}
\label{complex}
\end{center}
\end{figure}

\subsection{Some background}

Now we explain the definitions in Theorem~\ref{thmcat}.

Let $P$ and $K$ be (simplicial) complexes. A \textit{simplicial map} is a map $c \colon P \to K$ such that whenever $\Delta_P\subset P$ is a simplex then the image of the vertex set of $\Delta_P$ under $c$ is a set of vertices of $K$ that span a simplex $\Delta_K\subset K$. The map $c$ may also be thought of as a continuous map between the spaces $P$ and $K$: we may define $c$ on any simplex $\Delta_P \subset P$ by extending linearly the images of the vertices of $\Delta_P$ that span $\Delta_K\subset K$.

The \textit{$i$-skeleton} $K^{(i)}$ of $K$ is the unique subcomplex of $K$ consisting of all simplices of dimension at most $i$ in $K$.

A \textit{combinatorial loop} $c$ in $K$ is a sequence of vertices $(v_1,\ldots ,v_j)$ of $K^{(1)}$ where $v_j$ is adjacent (or equal) to $v_1$ and $v_i$ is adjacent (or equal) to $v_{i+1}$ whenever $1\leq i \leq j-1$. The \textit{combinatorial length} $l_C(c)$ of $c=(v_1,\ldots ,v_j)$ is equal to $j$. A combinatorial loop $c$ in $K$ may also be thought of as a simplicial map $c\colon P \to K$ where $P$ is a triangulation of $S^1$ with $j$ $1$-simplices.

We write $D^2$ for the closed unit disc with boundary $S^1$. Let $c$ be a combinatorial loop in $K$. We say that $c$ \textit{can be capped off with at most $n$ triangles} if there is a triangulation $P$ of $D^2$ into at most $n$ 2-simplices and there is a simplicial map $c'\colon P \to K^{(2)}$ such that $c'|_{S^1}=c$. In more informal words, the loop $c$ in $K$ can be deformed continuously to a point by pushing $c$ past at most $n$ triangles.

\begin{defn}\label{cii} Let $K$ be a simplicial complex. A function $f \colon \mathbb{N} \to \mathbb{N}$ is called a \textit{combinatorial isoperimetric bound} for $K$ if every combinatorial loop $c$ in $K$ can be capped off with at most $f(l_C(c))$ triangles. 

We say that $K$ satisfies a \textit{linear (or quadratic) combinatorial isoperimetric inequality} if there exists a combinatorial isoperimetric bound $f$ for $K$ such that $f(n)=O(n)$ (or $f(n)=O(n^2)$). We say that $K$ \textit{satisfies no combinatorial isoperimetric inequality} if no combinatorial isoperimetric bound of $K$ exists.
\end{defn}

\begin{rem} Another way of defining combinatorial isoperimetric inequalities is by using van Kampen diagrams, see \cite[Chapter~I.8A.4]{BridsonHaefliger}. Van Kampen diagrams are combinatorial maps from planar 2-complexes (these are not necessarily discs, they may have separating edges) to $K$. Notwithstanding it is well known that the van Kampen perspective and Definition~\ref{cii} are equivalent for complexes where the $2$-cells have boundedly many sides, in particular, in this paper the two notions agree. As we will see in Section~\ref{nocomb}, for our purposes it is more convenient to use Definition~\ref{cii}.
\end{rem}

\begin{rem}
Sometimes the existence of a combinatorial isoperimetric bound restricted to combinatorial loops of short length implies the existence of a bonafide combinatorial isoperimetric bound. For instance suppose that there exists $B$ such that whenever a combinatorial loop $c$ satisfies $l_C(c)\leq 16\delta$ then $c$ can be capped off using at most $B$ triangles. Furthermore, suppose that $K^{(1)}$ is $\delta$-hyperbolic (where each edge has length equal to $1$).  Then $K$ satisfies a linear combinatorial isoperimetric inequality. This follows from the linear coarse isoperimetric inequality (see \cite[Chapter~III.H~2.6~and~2.7]{BridsonHaefliger}). This is well known for Rips complexes of hyperbolic groups but the focus in this paper is on \textit{non-locally-compact} complexes.   Note that for the complex  in Figure~\ref{complex} there is no such integer $B$ despite $K^{(1)}$  being $2$-hyperbolic.
\end{rem}

Now we consider CAT(0) metrics on the topological spaces determined by simplicial complexes. We assume that the reader knows the definition of a CAT(0) metric space, see \cite[Chapter~II.1.1]{BridsonHaefliger}. For cosmetic reasons, in places where no confusion will arise, we abuse notation by writing $K$ for the simplicial complex, topological space, and the metric space $(K,d_K)$.

\begin{defn} \label{ssshapes} We say that $(K,d_K)$ has \textit{finitely many shapes} if there are only finitely many isometry types of metrics $d\colon \Delta \times \Delta \to \mathbb{R}$ on the simplices $\Delta$ of $K$ that are induced by restriction of $d_K$.

\end{defn}


We write $N_\epsilon(A)$ for the closed $\epsilon$-neighbourhood of a subset $A\subset K$. Likewise we write $N_\epsilon(x)=N_\epsilon(\{x\})$ for an element $x\in K$.

\begin{defn}\label{boundedthick}
We say that $K$ has \textit{bounded} shapes if for every $\delta>0$ there is a subdivision of each edge of $K$ into at most $n=n(K,\delta)$ intervals, such that each interval has diameter at most $\delta$.

We say that $K$ has \textit{thick} shapes if there exists $\epsilon>0$ such that the following holds. Write $\tilde{K}$ for the first barycentric subdivision of $K$. For each $v\in K^{(0)}$, set $N(v)$ to be the closure of the union of simplices $\tilde{\Delta}$ of $\tilde{K}$ such that $v\in \tilde{\Delta}$. Then for every simplex $\Delta$ of $K$, whenever $N_\epsilon(N(v))\cap \Delta\neq \emptyset$ then $v\in \Delta^{(0)}$ (see Figure~\ref{thick}).
\end{defn}

\begin{figure}
\begin{center}
\executeiffilenewer{thick.svg}{thick.pdf}%
{inkscape -z -D --file=thick.svg %
--export-pdf=thick.pdf --export-latex}%
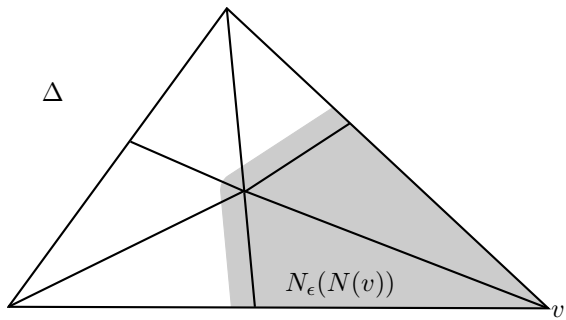%

\caption{The first barycentric subdivision of a $2$-simplex $\Delta$. A closed neighbourhood $N_\epsilon(N(v))$ of $N(v)$ is shaded and is disjoint from the opposite $1$-simplex; $\Delta$ is $\epsilon$-thick.}
\label{thick}
\end{center}
\end{figure}

In particular, having thick shapes means that disjoint simplices are uniformly far apart in $K$. Having bounded, thick shapes is a natural generalization of having finitely many shapes.

\begin{lem} \label{generalize}
If $K$ has finitely many shapes then $K$ has bounded, thick shapes.
\end{lem}

\begin{proof}[Sketch of proof] Clearly $K$ has bounded shapes. For thick shapes, for each simplex of $K$ we perform the first barycentric subdivision according to its isometry class, and then the finiteness provides some small enough $\epsilon>0$ to exist. \end{proof}

\begin{rem} \label{superremark}
It is an interesting question whether the complexes in Theorem~\ref{bigthm} admit a CAT(0) metric at all. If there is such a metric one can ask whether an interesting subgroup $H$ of the automorphism group still acts by isometries. By Theorem~\ref{bigthm} such a CAT(0) metric fails to have bounded, thick shapes, for example the edge lengths might be arbitrarily small (or large). In particular such a subgroup $H$ cannot have finite index.
\end{rem}

\subsection{Proof of Theorem~\ref{thmcat}}

The rest of this section is devoted to proving Theorem~\ref{thmcat}. For a point $x$ in the topological space $K$ we write $\Delta(x)$ for the minimal simplex of $K$ such that $x\in \Delta(x)$.

\begin{proof}[Proof of Theorem~\ref{thmcat}] There are different proofs but the proof given here determines an explicit way of capping off any combinatorial loop $c$ in $K$. We give a careful proof because $K$ is not necessarily locally compact.

Let $D$ and $\epsilon$ be constants such that $K$ has $D$-bounded, $\epsilon$-thick shapes.

The strategy of the proof is as follows. Given a combinatorial loop $c$ we abuse notation and write $c$ for the same topological loop in $K$, and its length in $(K,d_K)$, written $l_K(c)$, is bounded from above by $Dl_C(c)$. We parametrise $c$ by arc length in $K$ and so we write $c\colon [0,l_K(c)]\to K$. Without loss of generality $c(0)$ is a vertex of $K$. We subdivide the loop $c\colon [0,l_K(c)]\to K$ into points $c(t_i)$ at most $0.5\epsilon$ apart. For each $i$, we construct combinatorial paths $(v^i_j)_j$ connecting $v^i_0=c(0)$ to $v^i_{n(i)}$ where $c(t_i)\in N(v^i_{n(i)})$ (recall $N(v)$ from Definition~\ref{boundedthick}). The sequence $(v^i_{n(i)})_i$ defines a combinatorial loop in $K$, which is homotopic to the combinatorial loop $c$ (we discuss this at the very end of the proof). It suffices to cap off $(v^i_{n(i)})_i$ with a triangulation $P'$. We do this by constructing a triangulation $P(i)$ between $(v^i_j)_j$ and $(v^{i+1}_j)_j$ using linearly many triangles in terms of $l_C(c)$ (and there are at most linearly many $i$ in terms of $l_C(c)$). Gluing up the $P(i)$ gives a triangulation $P'$ that caps off $(v^i_{n(i)})_i$. Now we give the details.

Fix $0=t_0<t_1<\ldots <t_N=l_K(c)$ such that the closed intervals $[t_i,t_{i+1}]$ subdivide the closed interval $[0,l_K(c)]$ into pieces of length equal to $0.5\epsilon$ except possibly the last which has non-zero length at most $0.5\epsilon$. Thus $N=\lceil 2\epsilon^{-1}l_K(c) \rceil\leq 2\epsilon^{-1}Dl_C(c)+1$. We write $c_i\colon [0,d_K(c(0),c(t_i))]\to K$ for the geodesic (in the CAT(0) metric) parametrised by arc length that connects $c(0)$ and $c(t_i)$.

Now we construct explicit combinatorial paths that start at $c(0)$ by using the geodesics $c_i$. Fix $i$ and set $v^i_0=c(0)$ and $t^i_0=0$. Now for $j\geq1$ we define $v^i_j$ and $t^i_j$ inductively as follows. Note that $c_i(t^i_0)\in N(v^i_0)$ (recall $N(v)$ from Definition~\ref{boundedthick}), which is the base case.
\begin{itemize}

\item 

We set $t^i_j>t^i_{j-1}$ to be minimal such that $v^i_{j-1}\notin \Delta(c_i(t^i_j))$. Then we pick any $v^i_j\in K^{(0)}$ such that $c_i(t^i_j)\in N(v^i_j)$. By $\epsilon$-thick shapes we have $t^i_j-t^i_{j-1}\geq \epsilon$. We have $v^i_j$ adjacent to $v^i_{j-1}$.

\item 
On the other hand if no such $t^i_j$ above exists then $v^i_{j-1}$ is adjacent (or equal) to some vertex $v$ such that $c(t_i)\in N(v)$. We set $v^i_j=v$ and $t^i_j=d_K(c(0),c(t_i))$. Set $n(i)=j$. We stop.

\end{itemize}

We note that this inductive process must terminate: Whenever $0<j< n(i)$ then $t^i_j-t^i_{j-1}\geq \epsilon$ by $\epsilon$-thick shapes. Hence for $j<n(i)$ we have $t^i_j\geq j\epsilon$. Thus \[n(i)\leq \epsilon^{-1}d_K(c(0),c(t_i))+1 \leq \epsilon^{-1}Dl_C(c)+1.\]

Now for each $i$ such that $0\leq i \leq N-1$ we build triangles between the combinatorial paths $(v^i_j)_j$ and $(v^{i+1}_j)_j$. We use the comparison triangle $T(i)$ in $\mathbb{E}^2$ with sides $E_1$, $E_2$ and $E_3$ with side lengths $t^i_{n(i)}$, $d_K(c(t_i),c(t_{i+1}))$ and $t^{i+1}_{n(i+1)}$ respectively. Note that $d_K(c(t_i),c(t_{i+1}))\leq 0.5\epsilon$. We sweep out $T(i)$ using intervals of length at most $0.5\epsilon$ that are parallel to $E_2$. Write \[R_i=\frac{t^{i+1}_{n(i+1)}}{t^i_{n(i)}}.\]We may parametrise $E_1$ and $E_3$ by arc length such that $E_1(0)=E_3(0)$. Then the intervals parallel to $E_2$ are parametrised by their endpoints $E_1(s)$ and $E_3(R_is)$ for $0\leq s \leq t^i_{n(i)}$. In $K$ the intervals' endpoints correspond to $c_i(s)$ and $c_{i+1}(R_i s)$.

We endow $\partial T(i)$ with a simplicial structure induced from the points $E_1(t^i_j)$ and $E_3(t^{i+1}_j)$.

Now we define edges (written $e(m)$ below) between the  points $E_1(t^i_j)$ and the points $E_3(t^{i+1}_k)$ in order to triangulate $T(i)$. We justify doing this by finding edges (or proving equalities) between the appropriate vertices $v^i_j$ and $v^{i+1}_k$ in $K$. This is the hardest part of the proof and is the content of Lemmas~\ref{claim1}~and~\ref{claim2}. See Figure~\ref{ti} for an example triangulation.

\begin{figure}
\begin{center}
\executeiffilenewer{Ti.svg}{Ti.pdf}%
{inkscape -z -D --file=Ti.svg %
--export-pdf=Ti.pdf --export-latex}%
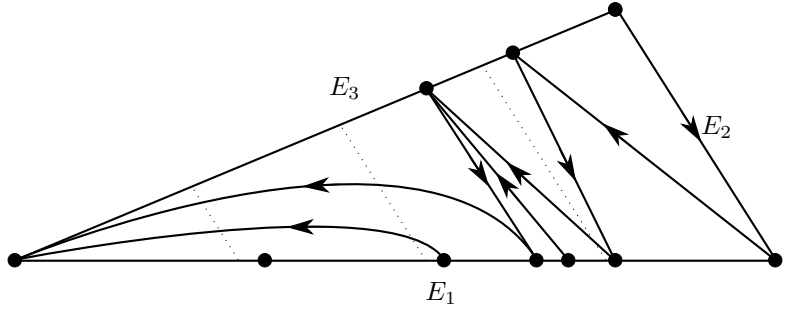%

\caption{An example (topological) triangulation $P(i)$ of $T(i)$. Here, $n(i)=6$, $n(i+1)=3$, and there are $8$ directed edges between $E_1$ and $E_3$. The edges pointing toward $E_1$ are constructed in Lemma~\ref{claim1}. The edges pointing toward $E_3$ are constructed in Lemma~\ref{claim2}. The first edge $e(1)$ always coincides with $E_2$. As we sweep out $T(i)$ starting from $E_2$, we encounter the edges $e(1),e(2),e(3)\ldots $ in that order.}
\label{ti}
\end{center}
\end{figure}

We lexicographically order the elements $(R_i t^i_j,i)$ and $(t^{i+1}_k,i+1)$ with the largest first. This ordering comes from the sweeping out of $T(i)$ by lines parallel to $E_2$: starting with $E_2$ itself and ending at the vertex opposite to $E_2$, we pass the vertices of $T(i)$ namely $E_1(t^i_j)$ and $E_3(t^{i+1}_k)$ and order them. We parametrise the above elements $(s(m),i(m))$ in order using a parameter $m$ where \[1\leq m\leq n(i)+n(i+1)+2.\] We introduce another parameter $j(m)$ so that the element $(s(m),i(m))$ will correspond to $v^{i(m)}_{j(m)}$ and $t^{i(m)}_{j(m)}$.

Now for \[1\leq m \leq n(i)+n(i+1)-1\] we find exactly one outward edge from (or another vertex equal to) $v^{i(m)}_{j(m)}$  using the following two lemmas. This enables us to define edges $e(m)$ that cut $T(i)$ into pieces that will give the required triangulation $P(i)$ of $T(i)$, see Figure~\ref{ti}.

\begin{lem}\label{claim1} If $i(m)=i+1$ then $v^i_j$ and $v^{i+1}_{j(m)}$ share an edge (or are equal) where $j$ is largest such that $t^i_j\leq R_i^{-1}t^{i+1}_{j(m)}$.
\end{lem}

\begin{proof}We argue that \[v^i_j,v^{i+1}_{j(m)}\in\Delta(c_i(R_i^{-1}t^{i+1}_{j(m)})).\]Using the comparison triangle $T(i)$ we have \[d_K(c_i(R_i^{-1}t^{i+1}_{j(m)}),c_{i+1}(t^{i+1}_{j(m)}))\leq 0.5\epsilon.\]Therefore we have  $N_\epsilon N(v^{i+1}_{j(m)})\cap \Delta(c_i(R_i^{-1}t^{i+1}_{j(m)}))\neq \emptyset$. By $\epsilon$-thick shapes we have $v^{i+1}_{j(m)}\in\Delta(c_i(R_i^{-1}t^{i+1}_{j(m)}))$.

Finally, we have $j$ largest such that $t^i_j\leq R_i^{-1}t^{i+1}_{j(m)}$. By definition of $v^i_j$ and $t^i_j$ we have $v^i_j\in\Delta(c_i(R_i^{-1}t^{i+1}_{j(m)}))$. \end{proof}

By Lemma~\ref{claim1}, as the corresponding vertices in $K$ are either equal or adjacent, we construct an edge between $E_1(t^i_j)$ and $E_3(t^{i+1}_{j(m)})$ in $T(i)$, see Figure~\ref{ti}. This defines the required $m$\textsuperscript{th} edge $e(m)$ inside $T(i)$  when $i(m)=i+1$.

However if $i(m)=i$ then we use

\begin{lem} \label{claim2}
 If $i(m)=i$ then $v^i_{j(m)}$ and $v^{i+1}_k$ share an edge (or are equal) where $k$ is largest such that $R_it^i_{j(m)}>t^{i+1}_k$.
\end{lem}
\begin{proof} We argue that for sufficiently small $\delta>0$ we have \[v^i_{j(m)},v^{i+1}_k\in\Delta(c_{i+1}(R_it^i_{j(m)}-\delta)).\]Indeed, for sufficiently small $\delta>0$ we have \[d_K(c_i(t^i_{j(m)}),c_{i+1}(R_it^i_{j(m)}-\delta))\leq \epsilon.\]Therefore we have $N_\epsilon N(v^i_{j(m)})\cap\Delta(c_{i+1}(R_it^i_{j(m)}-\delta))\neq\emptyset$ so by $\epsilon$-thick shapes we have $v^i_{j(m)}\in\Delta(c_{i+1}(R_it^i_{j(m)}-\delta))$. 

Finally for $\delta>0$ sufficiently small we have by definition \[v^{i+1}_k\in\Delta(c_{i+1}(R_it^i_{j(m)}-\delta))\] because $k$ is largest such that $R_it^i_{j(m)}>t^{i+1}_k$.\end{proof}

By Lemma~\ref{claim2}, as the corresponding vertices in $K$ are either equal or adjacent, we construct an edge between $E_1(t^i_{j(m)})$ and $E_3(t^{i+1}_k)$ in $T(i)$. This defines the required $m$\textsuperscript{th} edge $e(m)$ inside $T(i)$ for when $i(m)=i$, see Figure~\ref{ti}.

Now the edges $e(m)$ cut $T(i)$ into triangles that give a triangulation $P(i)$. This can be seen because the edges $e(m)$ and $e(m+1)$ share one vertex in common, and their endpoints correspond to values $t^i_j$ and $t^{i+1}_k$ that are monotonically decreasing in $m$. The existence of the edges in $K$ provide a natural map $P(i)\to K$ that extends the natural map $\partial T(i)\to K$. This is a simplicial map because $K$ is flag. We used exactly $n(i)+n(i+1)-1$ triangles, which is at most $2\epsilon^{-1}Dl_C(c)+1$.

Now we can cap off the combinatorial loop $(v^i_{n(i)})^{N-1}_{i=1}$ (it starts and ends at $c(0)$) with a triangulation $P'$ that is constructed by gluing (for each $i$) $P(i)$ to $P(i+1)$ along the combinatorial path $(v^{i+1}_j)^{n(i+1)}_{j=0}$. The number of triangles of $P'$ is at most $2\epsilon^{-1}Dl_C(c)+1$ multiplied by $N$. This is at most some quadratic function in $l_C(c)$.

\begin{figure}
\begin{center}\def\svgwidth{250pt}
\executeiffilenewer{lastbit.svg}{lastbit.pdf}%
{inkscape -z -D --file=lastbit.svg %
--export-pdf=lastbit.pdf --export-latex}%
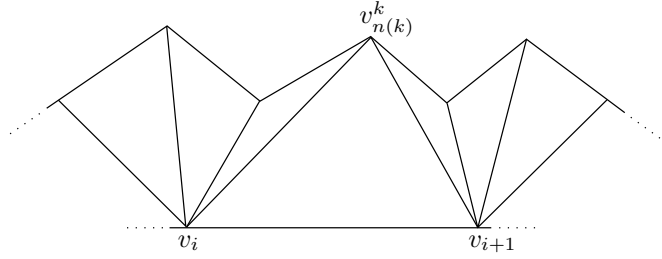%

\caption{The last step. Adding triangles to $P'$ to construct $P$ whose boundary maps onto the original combinatorial loop $c$.}
\label{last}
\end{center}
\end{figure}

Finally, we need to glue triangles along the boundary of $P'$ to construct a triangulation  $P$ that caps off the original combinatorial loop $c$. Write $c=(v_1,\ldots ,v_l)$ where $l=l_C(c)$. Each vertex $v_i$ is equal to several consecutive vertices \[v^j_{n(j)},v^{j+1}_{n(j+1)},\ldots ,v^k_{n(k)}.\] So we attach triangles between these for each $i$. Lastly there is a triangle spanned by $v_i$, $v_{i+1}$ and $v^k_{n(k)}$, see Figure~\ref{last}. This adds only at most $l_C(c)+N$ triangles.\end{proof}

\subsection{Higher-dimensional combinatorial isoperimetric inequalities}

In this section we prove that a CAT(0) flag simplicial complex with finitely many shapes must satisfy quadratic higher-dimensional combinatorial isoperimetric inequalities. The proof is versatile enough to state it in greater generality so we state Theorem~\ref{higherthm} below. The following definition is meant to be a $d$-dimensional analogue of the notion of bounded shapes, which was a constraint on the $1$-simplices of $K$.

\begin{defn}
Let $d\in\mathbb{N}$. We say that $(K,d_K)$ has $d$-\textit{bounded} shapes if given any $\delta>0$ there is a subdivision of the $d$-skeleton $K^{(d)}$ into a collection of smaller simplices $\Delta$ such that:-
\begin{itemize}
\item each $\Delta$ has diameter at most $\delta$, and,
\item every simplex of $K^{(d)}$ is subdivided into a uniformly bounded number of simplices depending only on $K$ and $\delta$.
\end{itemize}
\end{defn}

\begin{rem}If $K$ has finitely many shapes then $K$ has $d$-bounded, thick shapes for any $d\geq 1$.\end{rem}

Let $X$ be a (simplicial) complex. We write $CX=(X\times[0,1])/(X\times \{0\})$ for the \textit{cone} of $X$. We will abuse notation by writing $X=(X\times \{1\})$ for the corresponding subset of $CX$. When $X$ is $d$-dimensional we write $|X|$ for the number of $d$-simplices in $X$.

\begin{thm} \label{higherthm}
Fix $d>0$ and $K$ a (not necessarily locally compact) flag simplicial complex equipped with a CAT(0) metric with $d$-bounded, thick shapes. Then there exists a quadratic function $f\colon \mathbb{N} \to \mathbb{N}$ such that whenever $X$ is a connected $d$-dimensional complex and $g \colon  X \to K$ is a simplicial map then there is a triangulation $P$ of the cone $CX$ of $X$ and a simplicial map $g' \colon P \to K$ such that $g'|_{X}=g$, with $|P|\leq f(|X|^2)$.
\end{thm}

\begin{proof}
We follow the proof of Theorem~\ref{thmcat}, which is essentially a prerequisite. The first step is to subdivide $K^{(d)}$ into simplices of diameter at most $0.5\epsilon$, which is the only place in the proof where we require that $K$ is $d$-bounded. Such a subdivision pulls back to $X$, and this triangulation of $X$ we call $P''$, and there is a natural topological map $h\colon P''\to K$. For each vertex $u$ of $P''$ we pick $w\in K^{(0)}$ such that $h(u)\in N(w)$---this defines a map $g''\colon P''\to K$ by setting $g''(u)=w$ and by $\epsilon$-thickness we have that $g''$ is a simplicial map.

Let $w_0$ be an arbitrary vertex in the image (under $g$) of $X$ in $K$. We wish to find a triangulation $P'''$ of the cone $CX$ where the induced triangulation on $X\times \{1\}$ is $P''$, and have a simplicial map $g'''\colon P''' \to K$ extending $g''$. As a start we write $u_0$ for the vertex endpoint of the cone $CX$ (i.e. the point representing $X\times \{0\} / \sim$) and set $g'''(u_0)\coloneq w_0$.

We now want to construct many combinatorial paths between $u_0$ and $u\in P''$, and send them to combinatorial paths in $K$. We then fix a vertex $u$ of $P''$. We take the geodesic in the CAT(0) metric between $h(u_0)$ and $h(u)$. We define these combinatorial paths just as before in Theorem~\ref{thmcat}. Now let $u_1,\ldots,u_{d+1}$ span a simplex $\Delta$ in $P''$, we need to ensure that the combinatorial paths between $u_0$ and the $u_i\in \Delta$ extend to a triangulation of $C\Delta$ and a simplicial map $g'''_{C\Delta} \colon C\Delta \to K$. The construction and proof of this is a straightforward generalisation of Lemmas~\ref{claim1}~and~\ref{claim2}. This then provides a triangulation $P'''$ of $CX$ and a simplicial map $g'''\colon P''' \to K$. The number of $(d+1)$-simplices of $P'''$ is at most a uniform multiplicative constant multiple of $|X|$, multiplied by the diameter of $X$, which is again at most $|X|$, hence $O(|X|^2)$.

However $g'''$ is not necessarily equal to $g$ restricted to $X\times \{1\}$. Just as we did at the end of the proof of Theorem~\ref{thmcat} we add triangles to $P'''$ to find the required $P$ and $g'$. This adds on more $(d+1)$-simplices to $CX$, but at most some uniform constant multiple of $|X|$.\end{proof}

\section{Contractible, hyperbolic complexes with no combinatorial isoperimetric inequality} \label{nocomb}

In this section we prove the following theorem.

\begin{thm} \label{bigthm}
Let $K$ be one of the following complexes.
\begin{enumerate}
\item The arc complex $\mathcal{A}(S_{g,p})$ where $g\geq 2$ and $p\geq 2$, or, $g=1$ and $p\geq 4$, or $g=0$ and $p\geq 6$.
\item The disc complex $\mathcal{D}_n$ of a handlebody of genus $n\geq 5$.
\item The free splitting complex $\mathcal{FS}_n$ of a free group of rank $n\geq 5$.
\end{enumerate}
Then there is a family of loops $c_N$ of combinatorial length $4$ in $K^{(1)}$ such that the following holds. Whenever $P$ is a triangulation of a surface with one boundary component and $f\colon P \to K^{(2)}$ is a simplicial map where $f|_{\partial P}$ maps bijectively onto $c_N$ then $P$ must have at least $N$ triangles.
\end{thm}

Combining this with Theorem~\ref{thmcat} and Lemma~\ref{generalize} we have

\begin{cor}
Let $K$ be as in Theorem~\ref{bigthm}. Then $K$ does not admit a CAT(0) metric with finitely many shapes. Furthermore it does not admit one with bounded, thick shapes. \hfill $\square$
\end{cor}

\subsection{Most arc complexes}

We write $S_{g,p}$ for the orientable surface with genus $g$ and $p$ punctures/marked points.

\begin{proof}[Proof of Theorem~\ref{bigthm}(1)]

The hypothesis on $g$ and $p$ ensures that there exists an essential, non-peripheral simple closed curve $\gamma$ on $S=S_{g,p}$ such that $S-\gamma$ consists of two connected subsurfaces $Y$ and $Z$, both of which have at least one puncture (which is also a puncture of $S$), and $\xi(Y),\xi(Z)\geq 1$ (here $\xi(S_{g,p})=3g+p-3$, which coincides with the number of curves in a pants decomposition). Therefore $\mathcal{A}(Y)$ and $\mathcal{A}(Z)$ both have infinite diameter.

Pick an arbitrary integer $N\geq 2$. There exist arcs $a^Y_1$, $a^Y_2$, $a^Z_1$ and $a^Z_2$ of $S$ such that \begin{enumerate}
\item $a^Y_1,a^Y_2\subset Y$ and $a^Z_1,a^Z_2\subset Z$, and
\item $d_{\mathcal{A}(Y)}(a^Y_1,a^Y_2)\geq 3N$ and $d_{\mathcal{A}(Z)}(a^Z_1,a^Z_2)\geq 3N$.
\end{enumerate}

Note that the graph spanned by $a^Y_1$,  $a^Z_1$, $a^Y_2$ and $a^Z_2$ in $\mathcal{A}(S)$ is a loop $c_N$ whose combinatorial length is equal to $4$.

Suppose that there is a triangulation $P$ of a surface with one boundary component (for example a disc) and a simplicial map $f \colon P \to \mathcal{A}(S)^{(2)}$ such that $f|_{\partial P}$ coincides with $c_N$. Let us write $\tilde a^Y_i$ for $(f|_{\partial P})^{-1}a^Y_i$ and similarly $\tilde a^Z_i$ for $(f|_{\partial P})^{-1}a^Z_i$. Then $\tilde a^Y_1$, $\tilde a^Y_2$, $\tilde a^Z_1$ and $\tilde a^Z_2$ are the four vertices on the boundary of $P$.

We will colour the vertices of $\mathcal{A}(S)$ then pullback this colouring to give a colouring of the vertices of $P$. First we need some terminology. We say that an arc $a$ \textit{cuts} $Y$ if every representative of $a$ intersects $Y$. If $a$ cuts $Y$ then we can define $\kappa_Y(a)$ a simplex of $\mathcal{A}(Y)$ (see the definition of $\pi_Y'$ in \cite[p.~918]{MasurMinskyII}). If $a$ does not cut $Y$ then we say that $a$ \textit{misses} $Y$. If $a$ cuts $Y$ then we colour $a$ red. If $a$ is red then $\kappa_Y(a)$ is defined and it is a collection of essential arcs in $Y$. If $a$ misses $Y$ then we colour $a$ blue. If $a$ is blue then $a$ must cut $Z$ and therefore $\kappa_Z(a)$ is defined and it is a collection of essential arcs in $Z$. Each vertex is coloured either red or blue (but not both).

We claim that there is a red path between $\tilde a^Y_1$ and $\tilde a^Y_2$ in $P^{(1)}$ or a blue path between $\tilde a^Z_1$ and $\tilde a^Z_2$ in $P^{(1)}$. This follows directly from the proof of the $2$-dimensional Hex Theorem given in \cite[p.~820]{Gale}. We recall that beautiful proof here. We construct a graph $G$ with four vertices (illustrated as squares in Figure~\ref{hex}) with additional vertices that correspond to the triangles of $P$. Then we add edges to the vertices according to the rule depicted on the left of Figure~\ref{hex}; one imagines one colour as land and the other colour as the sea, and these edges represent the cliff between the two. Each vertex in $G$ has degree at most $2$ but there are precisely four vertices of degree $1$. Therefore there exists a path between two square vertices in $G$, see Figure~\ref{hex}. The required red or blue path can be constructed directly from this.

But by definition of $a^Y_1$, $a^Y_2$, $a^Z_1$ and $a^Z_2$, any such monochromatic path has length at least $3N$. This is because a red path of length $\ell$ can be used via the projection $\kappa_Y$ to construct a path of length at most $\ell$ in $\mathcal{A}(Y)$ between  $a^Y_1$ and $a^Y_2$, and similarly for a blue path in $\mathcal{A}(Z)$. Therefore there are at least $3N$ edges of $P$. If we count each triangle three times then we count each edge at least once, and so $P$ must have at least $N$ triangles. But $N\geq 2$ was arbitrary, and our loop had combinatorial length $4$, so we are done.
\end{proof}

\begin{figure}
\begin{center} \def\svgwidth{250pt}
\executeiffilenewer{hex.svg}{hex.pdf}%
{inkscape -z -D --file=hex.svg %
--export-pdf=hex.pdf --export-latex}%
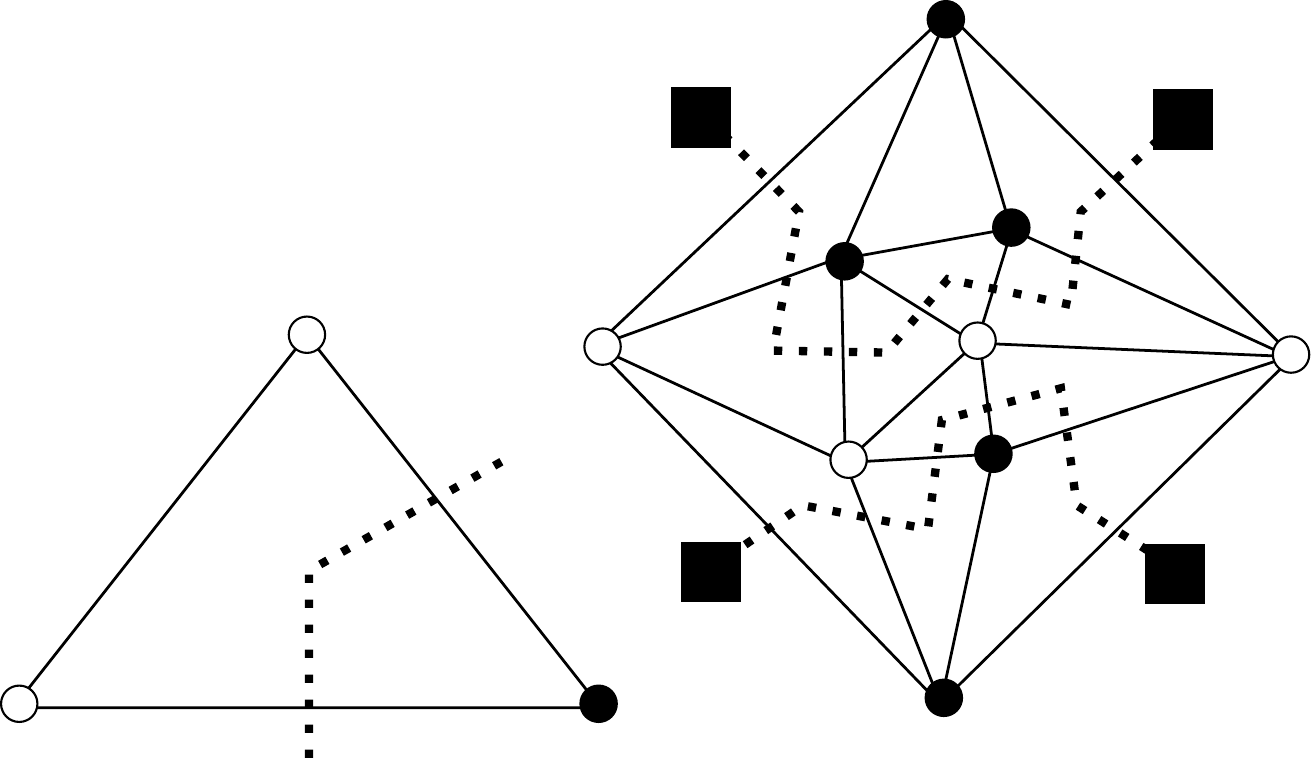%

\caption{Left: How the subgraph of the dual graph of $P$ is constructed locally. Right: An example construction.}
\label{hex}
\end{center}
\end{figure}

\subsection{All but finitely many disc complexes of handlebodies}

\begin{proof}[Proof of Theorem~\ref{bigthm}(2)]
The surface $S=S_{0,p}$ has Euler characteristic $1-(p-1)$, so the rank of its (free) fundamental group is $p-1$. The $3$-manifold $H=S\times [0,1]$ is a handlebody of genus $p-1$ and so its boundary $\partial H$ is a closed surface of genus $p-1$. Now let us set $n+1=p \geq 6$.

Given $N\geq 2$, we take the arcs $a^Y_1$, $a^Y_2$, $a^Z_1$ and $a^Z_2$ for $S=S_{0,p}$ as constructed in the proof of Theorem~\ref{bigthm}(1). The loop $c_N$ for the disc complex that we take is the one spanned by $a^Y_1 \times [0,1]$, $a^Y_2 \times [0,1]$, $a^Z_1 \times [0,1]$ and $a^Z_2 \times [0,1]$. Now we show that we require at least $N$ triangles for any such triangulation $P$ and simplicial map $f \colon P \to \mathcal{D}_n^{(2)}$ as in the statement of the theorem.

 Let $X$ be the subsurface of $\partial H$ corresponding to $S\times \{ 0\}$. Then $\kappa_X$ defines a map $\mathcal{D}_n^{(1)} \to \mathcal{A}(S)^{(1)}$ via $D\mapsto \kappa_X \partial D$, which is defined for all discs $D\in\mathcal{D}_n$ because the boundary $\partial D$ of every essential disc $D$ cuts $X=S\times\{0\}$. If $D_1$ and $D_2$ were disjoint then so are $\kappa_{X}\partial D_1$ and $\kappa_{X}\partial D_2$, and so by taking an arbitrary arc in each $\kappa_{X}\partial D$ we may define a $1$-Lipschitz map from $\mathcal{D}_n^{(1)}$ to $\mathcal{A}(S)^{(1)}$, which we call $g$.

Note that $g$ sends the loop $c_N$ to the original loop spanned by $a^Y_1$, $a^Y_2$, $a^Z_1$ and $a^Z_2$ in the arc complex.

Therefore given $f \colon P \to \mathcal{D}_n^{(2)}$ the composition $gf \colon P \to \mathcal{A}(S)^{(2)}$ provides us with a simplicial map where $(gf)|_{\partial P}$ coincides with the original loop of length $4$ in the arc complex. Therefore $P$ has at least $N$ triangles as in the proof of Theorem~\ref{bigthm}(1). \end{proof}

\subsection{All but finitely many free splitting complexes of free groups} \label{sec:freesplit}

\begin{proof}[Proof of Theorem~\ref{bigthm}(3)]

We require a fascinating result of Hamenst\"adt~and~Hensel.

\begin{prop}[Proposition~4.18 of \cite{HamenstadtHensel}] \label{hh} Let $S$ be a compact surface with genus $g$ and $|\partial S|=b$. There is a canonical $1$-Lipschitz embedding $\mathcal{A}(S) \to \mathcal{FS}_n$ where $n=2g+b-1$, and there exists a $1$-Lipschitz left inverse $\mathcal{FS}_n \to \mathcal{A}(S)$.
\end{prop}

It is worth remarking that while the embedding is canonical their $1$-Lipschitz left inverse is not because it depends on a choice of maximal arc system of $S$. For our purposes a $1$-Lipschitz left inverse is crucial. It is important to note that canonical \textit{coarsely} Lipschitz left inverses to this map have been given by Bowditch~and~Iezzi \cite{BowditchIezzi} and by Forlini \cite{Forlini} (who also addresses the arc-and-curve complex and the cyclic splitting complex).

We note that \cite[Proposition~4.18]{HamenstadtHensel} is formally stated only for the case $b=1$. It is remarked \cite[Remark~4.1]{HamenstadtHensel} that minor modifications should enable the general case that we require here. Another proof via their ideas can be given using a fixed pants decomposition of $S$, and then defining a kind of tight minimal position on each pair of pants.

The 1-Lipschitz left inverse enables us to argue analogously to that of the case of the disc complex. We take the loops from Theorem~\ref{bigthm}(1) and embed them into $\mathcal{FS}_n$. Any disc bounding such a loop can be mapped via Proposition~\ref{hh} back into $\mathcal{A}(S)$, so we are done. \end{proof}

\section{Complexes satisfying a linear combinatorial isoperimetric inequality} \label{linearsec}

In this section we prove

\begin{thm} \label{linearcombi} Let $K$ be one of the following complexes. \begin{enumerate}
\item The curve complex $\mathcal{C}(S)$ with $S=S_{g,p}$ and $3g+p-3\geq 2$.
\item The arc-and-curve complex $\mathcal{AC}(S)$.
\end{enumerate}
Then $K$ satisfies a linear combinatorial isoperimetric inequality.
\end{thm}

\begin{rem} We only treat the case where $\xi(S)=3g+p-3 \geq 2$ because the omitted cases are either straightforward or well known.  It is a theorem of Harer \cite{Harer} that the curve complex is homotopy equivalent to an infinite wedge of spheres of fixed dimension depending on $g$ and $p$. The only cases we consider in which the original curve complex $\mathcal{C}(S)$ is not simply connected is that of the five-times-punctured sphere $S_{0,5}$ and the two-times-punctured torus $S_{1,2}$. Nonetheless in these cases one can attach pentagons to all (isometrically embedded) loops of length $5$, and this creates a simply-connected (but non-contractible) complex, which is due to Piotr~Przytycki. Note that such loops of length $5$ are unique up to $\mathrm{Mod}(S)$, so cocompactness is preserved. The case of Przytycki's complex and the arc-and-curve complex will be explained in Section~\ref{AC}.
\end{rem}

\begin{proof}[Proof of Theorem~\ref{linearcombi}(1)] Let us discuss the strategy and outline of the proof and then defer the omitted details to Sections~\ref{sectighten}~and~\ref{secfinite}.

The curve graph $\mathcal{C}(S)^{(1)}$ is $\delta$-hyperbolic \cite{MasurMinskyI} and therefore satisfies a linear coarse isoperimetric inequality \cite[Proposition~III.H.2.7]{BridsonHaefliger}. In other words, given an arbitrary loop in $\mathcal{C}(S)^{(1)}$ we may decompose it into a net of linearly many loops of length at most $16\delta$ i.e. linearly many short loops. To prove the theorem, it suffices to show that there exists an a priori bound on the number of triangles required to cap off a short loop. This is the content of the proof. It is not straightforward; the arc complex generally fails to have this property even for loops of length $4$, see Theorem~\ref{bigthm}(1).

We show that after a uniformly bounded amount of homotoping a short loop (i.e. a homotopy across a bounded number of triangles) there is either a shortcut available that divides the loop into two/three smaller ones (then use induction on the length), or, the loop $(\gamma_i)_i$ is a \textit{subpath} of a short, \textit{tight loop} $(C_i)_i$, see Lemma~\ref{lem:wiggle}.

In Section~\ref{secfinite} we show that there are only finitely many short, tight loops $c=(C_i)_i$ up to the action of $\mathrm{Mod}(S)$, see Theorem~\ref{thm:finitetight}. Let $c'=(\gamma_i)_i$ be a \textit{subpath} of $c$ i.e. $\gamma_i$ is a component of $C_i$ for each $i$. Then there are only finitely many $\mathrm{Mod}(S)$-orbits of such $(\gamma_i)_i$ too. Since $\mathcal{C}(S)$ is simply connected, we have that each such $(\gamma_i)_i$ has some way of being capped off with a disc, but there are only finitely many orbits of such loop, and so an a priori bound on the number of triangles required exists and the proof is complete.\end{proof}

\subsection{Homotoping and simplifying a loop via the tightening procedure} \label{sectighten}

Masur~and~Minsky introduced the notion of a tight geodesic and showed that they exist between any pair of vertices in the curve complex (see \cite[Lemma~4.5]{MasurMinskyII}, they were originally called \textit{tight sequences}). We refer to the idea in their proof as the tightening procedure. As we will see, the procedure is slightly more complicated for loops than it is for geodesics. Such a procedure for loops is new.

First we define what we mean by tightening a sequence, and then we give some remarks to make the definition clearer and explain the purpose.

\begin{defn} Suppose that $(C_i)_i$ is a sequence of curve systems (or multicurves) of $S$, such that $C_i$ misses $C_{i+1}$, and such that each component $\gamma_i$ of $C_i$ cuts each component $\gamma_{i+2}$ of $C_{i+2}$. We say that $(C_i)_i$ is \textit{tight} at $j$ if the curve system $C_j$ consists precisely of the essential curves of the boundary of a closed regular neighbourhood of $C_{j-1}\cup C_{j+1}$. We will simply write $\partial (C_{j-1}, C_{j+1})$ for this curve system. If $(C_i)_i$ is not tight at $j$ then we may \textit{tighten} the sequence at $j$ by replacing $C_j$ by $\partial (C_{j-1}, C_{j+1})$.
\end{defn}

\begin{rem} \label{remtighten} \leavevmode \begin{enumerate}[(1)]
\item  The reader may wonder why we are discussing curve systems and not simply curves. If one starts with a geodesic in the curve graph and then starts tightening the sequence then we might replace a curve by a curve system, and so eventually we may have to consider the process of tightening using curve systems also.
\item Our assumption that $C_i$ misses $C_{i+1}$ simply means that the curve systems admit disjoint representatives on $S$. One should compare this to paths in the curve graph. On the other hand, if two isotopy classes do not miss then we say that they \textit{cut}. 
\item We want each component $\gamma_i$ of $C_i$ to cut each component $\gamma_{i+2}$ of $C_{i+2}$, otherwise there is an obvious shortcut from $\gamma_i$ to $\gamma_{i+2}$. Such a shortcut can then be used to cut up our original loop into smaller ones.
\item The curve system $\partial (C_{j-1}, C_{j+1})$ is defined by first taking representatives of $C_{j-1}$ and $C_{j+1}$ that intersect transversely and minimally. Then the set $C_{j-1}\cup C_{j+1}$ is well defined on $S$ up to ambient isotopy (this is well known see for example \cite[Lemma~2.2]{WebbCombinatorics} for a proof). We can then discuss the closed regular neighbourhood $N=N(C_{j-1}\cup C_{j+1})$. Its boundary has at least one essential curve of $S$ because if it did not then $S-N$ would be a collection of open discs and open peripheral annuli, but $C_j$ is an essential curve system of $S$ contained inside $S-N$, a contradiction. We write $\partial (C_{j-1}, C_{j+1})$ for the (non-empty) curve system obtained by taking all such essential curves of $\partial N$.
\item We have that $C_j$ misses $C_j'=\partial(C_{j-1}, C_{j+1})$, and so when we tighten the sequence $(C_i)_i$ at $j$ we are in fact homotoping $C_j$ across a simplex in $\mathcal{C}(S)$ to $C_j'$. Therefore if we take arbitrary components $\gamma_i$ of $C_i$ and $\gamma_j'$ of $C_j'$ then we are homotoping $\gamma_j$ across two triangles to $\gamma_j'$. One of the triangles is spanned by $\gamma_{j-1}$, $\gamma_j$, and $\gamma_j'$; the other is $\gamma_{j+1}$, $\gamma_j$, and $\gamma_j'$. This fact is used in Lemma~\ref{lem:wiggle}.
\item The length of the sequence does not change. Furthermore, in an intuitive but not mathematical sense, $\partial (C_{j-1}, C_{j+1})$ is at most as complicated as $C_j$, and so tightening a sequence is a way of simplifying it.
\end{enumerate}
\end{rem}

Before we treat the harder and more general case let us engage with the loops of length $4$ first to get a feel of where this proof is going. Note that any loop of length $3$ already bounds a triangle in $\mathcal{C}(S)^{(2)}$, and any length less than this is straightforward.

\begin{lem}
Let $c$ be a loop of length $4$ in $\mathcal{C}(S)^{(1)}$. Then $c$ can be capped off with a disc of $2$ or $4$ triangles.
\end{lem}

\begin{proof}
We write $c=(c_i)_i$ where $i$ ranges over the integers modulo $4$. We may assume the vertices are distinct. If $c_0$ and $c_2$ miss then there is an edge between them and so $c_0$, $c_1$, $c_2$ form a triangle and so does $c_2$, $c_3$, $c_4=c_0$, and we are done. Similarly we are done with two triangles if $c_1$ and $c_3$ are disjoint. Notice that we are dividing $c$ into two loops of length $3$ here via a shortcut.

Suppose instead that there are no such shortcuts. The plan instead is to consider tightening. Write $C=\partial(c_0, c_2)$. Then we observe that $C$ misses $c_1$, see Remark~\ref{remtighten}(4)~and~(5), furthermore $C$ also misses $c_3$.

We now give an awkward end to the proof for the purpose of indicating how it generalises to larger lengths. We tighten $c$ at the index $1$ to obtain a loop of curve systems $(c_0, C, c_2, c_3)$. We observe that $C$ misses $c_3$. This is a shortcut, so we should aim to divide the loop into two. Let $\gamma$ be an arbitrary component of $C$. Then we may homotope $c$ across two triangles in $\mathcal{C}(S)^{(2)}$ to obtain the loop $c'=(c_0,\gamma,c_2,c_3)$; the triangles mentioned are bounded by $c_0$, $c_1$ and $\gamma$, and, $c_1$, $c_2$ and $\gamma$. But $\gamma$ misses $c_3$ and so we are in the case with a shortcut, and so $c'$ can be capped off with $2$ triangles, and in turn, $c$ can be capped off with $4$ triangles.\end{proof}

In the proof above we had to consider a loop of curve systems. From now on our short loop can be assumed to be a (periodic) sequence of curve systems $(C_i)_i$ where $i$ ranges over the integers modulo $n$, and $n$ is the length of the loop. Of course we are assuming that $C_i$ misses $C_{i+1}$.

Let us make the following definition to tidy up the upcoming statement of Lemma~\ref{lemtighten}.

\begin{defn} Let $c=(C_i)_i$ be a loop of curve systems of length $n\geq 5$. A \textit{shortcut} for $c$ is any one of the following
\begin{enumerate}[(1)]
\item a component $\gamma_j$ of $C_j$ and a component $\gamma_{j+2}$ of $C_{j+2}$ with $d_{\mathcal{C}(S)}(\gamma_j,\gamma_{j+2})<2$,  or,
\item a component $\gamma_j$ of $C_j$ and a component $\gamma_{j+3}$ of $C_{j+3}$ with $d_{\mathcal{C}(S)}(\gamma_j,\gamma_{j+3})<3$, and $n\geq 6$, or,
\item a curve $\gamma$ adjacent to $\gamma_{j-1}$, $\gamma_{j+1}$, and $\gamma_{j+2}$, some components of $C_{j-1}$, $C_{j+1}$, and $C_{j+2}$ respectively.
\end{enumerate}
\end{defn}

We require $n\geq 6$ in the above second case for the following reason. If one has a loop of length $6$ and finds a shortcut of the second type described above, then one can divide the loop into two smaller loops, and then we are done by induction---this doesn't quite work for $n=5$. The third type is only introduced for handling the case $n=5$ and enables us to divide the pentagon into two squares and one triangle. Of course the third type is a strong case of the second type when $n\geq 6$.

\begin{defn} \label{deftightloop} Let $c=(C_i)_i$ be a loop of curve systems of length $n\geq 5$. We say that $c$ is \textit{tight} if
\begin{enumerate}
\item there are no shortcuts for $c$, and,
\item $c$ is tight at all indices.
\end{enumerate}
\end{defn}

We will see in Section~\ref{secfinite} that there are only finitely many tight loops of a given length up to the action of the mapping class group.

\begin{lem}[Tightening procedure for loops] \label{lemtighten}
Let $c=(C_i)_i$ be a loop of curve systems of length $n\geq 5$. Then there is a way of tightening $c$ at most $n$ times, during or at the end of which we either have that
\begin{enumerate}
\item there is a shortcut for $c$, or,
\item $c$ is a tight loop of length $n$.
\end{enumerate}

\end{lem}

\begin{proof}

We may assume that $c$ has no shortcuts throughout.



The proof requires a discussion of the \textit{subsurfaces filled by $C_{i-1}$ and $C_{i+1}$}, written $F(C_{i-1},C_{i+1})$, and their behaviour under the tightening procedure. This subsurface is defined to be the closure of: the union of a regular neighbourhood of $C_{i-1}\cup C_{i+1}$ union the complementary discs and once-punctured discs.

We tighten at index $i$, so we can assume $C_i=\partial (C_{i-1}\cup C_{i+1})$. Then we attempt to tighten at index $i+1$ by replacing $C_{i+1}$ with $C_{i+1}'=\partial(C_i\cup C_{i+2})$. We wish to keep being tight at $i$. We may assume that there are no shortcuts and so every component of $C_{i+1}'$ cuts every component of $C_{i-1}$. Therefore $C_{i+1}'$ is contained in $F(C_{i-1},C_{i+1})$ because it misses the boundary $C_i$. Hence $F(C_{i-1},C_{i+1}')\subset F(C_{i-1},C_{i+1})$. Now if this is a strict inclusion then there exists $\gamma$ that misses $C_{i-1}$ and $C_{i+1}'$ but is contained in $F(C_{i-1},C_{i+1})$. Therefore $\gamma$ cuts $C_{i+1}$ because $C_{i-1}$ and $C_{i+1}$ fill $F(C_{i-1},C_{i+1})$. If $\gamma$ misses a component of $C_{i+2}$ then there is a shortcut of the third type, namely via $C_{i-1}$, $C_{i+1}'$, and $C_{i+2}$, and we are done. So instead $\gamma$ cuts every component of $C_{i+2}$. Then $\gamma$ must be contained in $F(C_i,C_{i+2})$ because it misses the boundary $C_{i+1}'$. But then $\gamma$ also misses $C_{i+1}$, a contradiction. Therefore if we assume there are no shortcuts throughout then after $n$ tightenings we obtain a tight loop of length $n$.\end{proof}

We now obtain

\begin{lem} \label{lem:wiggle} Given any loop $c=(\gamma_i)_i$ of length $n\geq 5$, after a homotopy of $c$ past at most $2n$ triangles, we either find a shortcut for $c$ or it is a subpath of a tight loop of length $n$.
\end{lem}

\begin{proof} We follow Lemma~\ref{lemtighten} starting with $c$. This lemma states that after at most $n$ tightenings of the loop we either find a shortcut or the loop becomes a tight loop of length $n$. Suppose first that no shortcuts arise. Then for each curve system throughout the procedure, pick an arbitrary component. Then each loop of curve systems corresponds to a loop of curves, and whenever we tighten, then our loop of curves either stays the same or is homotoped past two triangles. Hence after homotoping $c$ past at most $2n$ triangles we obtain the required subpath.

On the other hand if there is a shortcut during the tightening process, then we use those components of those curve systems that provide the shortcut. This then provides a homotopy of $c$ past at most $2n$ triangles ending up with a shortcut for the loop of curves. \end{proof}

\subsection{Only finitely many kinds of short, tight loop}  \label{secfinite} In this section we adapt the argument in \cite[Theorem~4.7]{WebbCombinatorics} to show

\begin{thm} \label{thm:finitetight} For every surface $S$ with $\xi(S)\geq 2$ and $n\geq 5$ there are only finitely many $\mathrm{Mod}(S)$-orbits of tight loop of length $n$.
\end{thm}

\begin{proof} Let $\kappa_{C_i}(C)$ be the arc system (not counting parallel copies) of $S-C_i$ that is determined by the curve system $C$, defined when each component of $C$ cuts $C_i$. The key observation that makes this proof work is, due to being tight at $i$ (and having no shortcuts), that the arc system $\kappa_{C_{i-1}}(C_{i+1})$ determines $C_i$ \cite[Lemma~4.4]{WebbCombinatorics}.

We first claim that the geometric intersection number $i(C_{i-1},C_{i+1})$ is bounded above in terms of $S$ and $n$. So suppose that $n\geq 6$. Then $C_{-3}$ and $C_0$ fill $S$. Up to the mapping class group, there are only finitely many possibilities for $\kappa_{C_0}(C_{-3})$. Because this fills, we have only finitely many possibilities for $\kappa_{C_0}(C_{-2})$. By considering $\kappa_{C_0}(C_j)$ for $j=-3,-4,\ldots, 2$ (modulo $n$) we also have only finitely many possibilities for $\kappa_{C_0}(C_2)$. But these determine $C_{-1}$ and $C_1$, and so our first claim follows.

Now suppose that $n=5$. Then $C_{i-2}\cup C_{i+2}$ is a curve system and $\kappa_{C_i}(C_{i-2}\cup C_{i+2})$ fills $S-C_i$, because there is no shortcut. Then after a mapping class we can assume that the pair $\kappa_{C_i}(C_{i-2})$ and $\kappa_{C_i}(C_{i+2})$ is one of finitely many possibilities. But this then determines the pair $C_{i-1}$ and $C_{i+1}$, so we are done as before. In fact, this settles the theorem in the case $n=5$ by bounding the geometric intersection number of the collection of curve systems. So we assume $n\geq 6$ for the remainder of the proof.

Our second claim is that $i(C_{i-1},C_{i+2})$ is bounded above in terms of $S$ and $n$. Start off with $\kappa_{C_i}(C_{i-3})$, which intersects $C_{i-1}$ a uniformly bounded number of times by the first claim above. Then because $\kappa_{C_i}(C_{i-3})$ fills $S-C_i$, we can use an inductive argument with $j=i-3,i-4,\ldots, i+2$ (modulo $n$), to see that $i(\kappa_{C_i}(C_{i+2}),C_{i-1})$ is bounded above in terms of $S$ and $n$. By our first claim $C_{i+2}$ determines only a uniformly bounded number of parallel copies of $\kappa_{C_i}(C_{i+2})$, and so our second claim follows.

Finally, there are only finitely many $\mathrm{Mod}(S)$-orbits of the pair $C_0$ and $C_3$ by our second claim above. So fix $C_0$ and $C_3$ to be one such pair. But then we obtain two different paths of curve systems between $C_0$ to $C_3$. These are both so-called \textit{tight filling multipaths} of length at most $n$, see \cite[Section~3]{WebbCombinatorics}. But there are only finitely many such paths connecting $C_0$ to $C_3$ in terms of $S$ and $n$, see \cite[Theorem~4.7]{WebbCombinatorics} (or \cite[Appendix~A]{BirmanMM} for an exposition). Briefly speaking the idea of that proof is to consider $\kappa_{C_0}(C_i)$. Whenever this fills then there are only finitely many possibilities for $\kappa_{C_0}(C_{i+1})$ (and $\kappa_{C_0}(C_{i-1})$). However $\kappa_{C_0}(C_{2})$ (and $\kappa_{C_0}(C_{-2})$) do not fill, but in this case $C_1$ (and $C_{-1}$) are determined by those arc systems. We know precisely what $C_3$ is, so we can also deduce the finitely many possibilities for $C_2$ and $C_4$, and so on. \end{proof}

\subsection{The arc-and-curve complexes} \label{AC}

Now we show that Theorem~\ref{linearcombi}(2) follows from Theorem~\ref{linearcombi}(1).

\begin{proof}[Proof of Theorem~\ref{linearcombi}(2)]

The arc-and-curve complex is quasi-isometric to $\mathcal{C}(S)$, and so it is also hyperbolic, therefore once again it suffices to show that for any loop of vertices $c=(v_i)_i$ of bounded length, there is a bounded number of triangles required for a disc to cap off $c$ to deduce the theorem.

Let $c$ have length $n$. Our first step is to homotope $c$ into $\mathcal{C}(S)$, but across at most $3n$ triangles, and to a loop of length at most $2n$. This argument is well known. To see this define $c_i\in\mathcal{C}(S)$ for each vertex $v_i$ of $c$. If $v_i$ is a curve we set $c_i=v_i$. If $v_i$ is an arc then we set $c_i$ to be one of the peripheral curves of $S-v_i$ which is essential in $S$. Thus $v_i$ and $c_i$ are either equal or adjacent. Moreover for each $i$ we can pick a curve $c_i'$ which is peripheral in $S-v_i-v_{i+1}$ but which is essential in $S$. Such a curve $c_i'$ is adjacent (or equal) to $v_i$, $c_i$, $v_{i+1}$, and $c_{i+1}$. Therefore we can push $c$ past at most $3n$ triangles into $\mathcal{C}(S)$ to a loop of length at most $2n$.

We can then use Theorem~\ref{linearcombi}(1), unless $S$ is $S_{0,5}$ or $S_{1,2}$, in which case our loop of curves is capped off by boundedly many pentagons. But isometrically embedded pentagons are unique in $\mathcal{C}(S)$, so because $\mathcal{AC}(S)$ is contractible and therefore simply connected, there is a uniform upper bound on the triangles required. The theorem then follows.\end{proof}

\bibliography{arccatreferences}
\bibliographystyle{alpha}

\end{document}

%% file: complex.pdf_tex
\begingroup%
  \makeatletter%
  \providecommand\color[2][]{%
    \errmessage{(Inkscape) Color is used for the text in Inkscape, but the package 'color.sty' is not loaded}%
    \renewcommand\color[2][]{}%
  }%
  \providecommand\transparent[1]{%
    \errmessage{(Inkscape) Transparency is used (non-zero) for the text in Inkscape, but the package 'transparent.sty' is not loaded}%
    \renewcommand\transparent[1]{}%
  }%
  \providecommand\rotatebox[2]{#2}%
  \ifx\svgwidth\undefined%
    \setlength{\unitlength}{291.31342773bp}%
    \ifx\svgscale\undefined%
      \relax%
    \else%
      \setlength{\unitlength}{\unitlength * \real{\svgscale}}%
    \fi%
  \else%
    \setlength{\unitlength}{\svgwidth}%
  \fi%
  \global\let\svgwidth\undefined%
  \global\let\svgscale\undefined%
  \makeatother%
  \begin{picture}(1,0.275243)%
    \put(0,0){\includegraphics[width=\unitlength]{complex.pdf}}%
    \put(-0.0024941,0.13793368){\color[rgb]{0,0,0}\makebox(0,0)[lb]{\smash{...}}}%
    \put(0.97952435,0.13793368){\color[rgb]{0,0,0}\makebox(0,0)[lb]{\smash{...}}}%
    \put(0.2055607,0.22670372){\color[rgb]{0,0,0}\makebox(0,0)[lb]{\smash{...}}}%
    \put(0.80022346,0.22670372){\color[rgb]{0,0,0}\rotatebox{-180}{\makebox(0,0)[lb]{\smash{...}}}}%
    \put(0.80022346,0.04995683){\color[rgb]{0,0,0}\rotatebox{-180}{\makebox(0,0)[lb]{\smash{...}}}}%
    \put(0.2055607,0.05134384){\color[rgb]{0,0,0}\makebox(0,0)[lb]{\smash{...}}}%
  \end{picture}%
\endgroup%

%% file: thick.pdf_tex
\begingroup%
  \makeatletter%
  \providecommand\color[2][]{%
    \errmessage{(Inkscape) Color is used for the text in Inkscape, but the package 'color.sty' is not loaded}%
    \renewcommand\color[2][]{}%
  }%
  \providecommand\transparent[1]{%
    \errmessage{(Inkscape) Transparency is used (non-zero) for the text in Inkscape, but the package 'transparent.sty' is not loaded}%
    \renewcommand\transparent[1]{}%
  }%
  \providecommand\rotatebox[2]{#2}%
  \ifx\svgwidth\undefined%
    \setlength{\unitlength}{217.49875488bp}%
    \ifx\svgscale\undefined%
      \relax%
    \else%
      \setlength{\unitlength}{\unitlength * \real{\svgscale}}%
    \fi%
  \else%
    \setlength{\unitlength}{\svgwidth}%
  \fi%
  \global\let\svgwidth\undefined%
  \global\let\svgscale\undefined%
  \makeatother%
  \begin{picture}(1,0.54027872)%
    \put(0,0){\includegraphics[width=\unitlength]{thick.pdf}}%
    \put(0.0596391,0.378176){\color[rgb]{0,0,0}\makebox(0,0)[lb]{\smash{$\Delta$}}}%
    \put(0.94240271,0.00378954){\color[rgb]{0,0,0}\makebox(0,0)[lb]{\smash{$v$}}}%
    \put(0.48131643,0.04845327){\color[rgb]{0,0,0}\makebox(0,0)[lb]{\smash{$N_\epsilon(N(v))$}}}%
  \end{picture}%
\endgroup%

%% file: Ti.pdf_tex
\begingroup%
  \makeatletter%
  \providecommand\color[2][]{%
    \errmessage{(Inkscape) Color is used for the text in Inkscape, but the package 'color.sty' is not loaded}%
    \renewcommand\color[2][]{}%
  }%
  \providecommand\transparent[1]{%
    \errmessage{(Inkscape) Transparency is used (non-zero) for the text in Inkscape, but the package 'transparent.sty' is not loaded}%
    \renewcommand\transparent[1]{}%
  }%
  \providecommand\rotatebox[2]{#2}%
  \ifx\svgwidth\undefined%
    \setlength{\unitlength}{291.78452148bp}%
    \ifx\svgscale\undefined%
      \relax%
    \else%
      \setlength{\unitlength}{\unitlength * \real{\svgscale}}%
    \fi%
  \else%
    \setlength{\unitlength}{\svgwidth}%
  \fi%
  \global\let\svgwidth\undefined%
  \global\let\svgscale\undefined%
  \makeatother%
  \begin{picture}(1,0.39025226)%
    \put(0,0){\includegraphics[width=\unitlength]{Ti.pdf}}%
    \put(0.41485033,0.27132885){\color[rgb]{0,0,0}\makebox(0,0)[lb]{\smash{$E_3$}}}%
    \put(0.5394817,0.0054487){\color[rgb]{0,0,0}\makebox(0,0)[lb]{\smash{$E_1$}}}%
    \put(0.89537336,0.22009149){\color[rgb]{0,0,0}\makebox(0,0)[lb]{\smash{$E_2$}}}%
  \end{picture}%
\endgroup%

%% file: lastbit.pdf_tex
\begingroup%
  \makeatletter%
  \providecommand\color[2][]{%
    \errmessage{(Inkscape) Color is used for the text in Inkscape, but the package 'color.sty' is not loaded}%
    \renewcommand\color[2][]{}%
  }%
  \providecommand\transparent[1]{%
    \errmessage{(Inkscape) Transparency is used (non-zero) for the text in Inkscape, but the package 'transparent.sty' is not loaded}%
    \renewcommand\transparent[1]{}%
  }%
  \providecommand\rotatebox[2]{#2}%
  \ifx\svgwidth\undefined%
    \setlength{\unitlength}{426.62148438bp}%
    \ifx\svgscale\undefined%
      \relax%
    \else%
      \setlength{\unitlength}{\unitlength * \real{\svgscale}}%
    \fi%
  \else%
    \setlength{\unitlength}{\svgwidth}%
  \fi%
  \global\let\svgwidth\undefined%
  \global\let\svgscale\undefined%
  \makeatother%
  \begin{picture}(1,0.35672671)%
    \put(0,0){\includegraphics[width=\unitlength]{lastbit.pdf}}%
    \put(0.26240334,0.00394634){\color[rgb]{0,0,0}\makebox(0,0)[lb]{\smash{$v_i$}}}%
    \put(0.70186621,0.00394634){\color[rgb]{0,0,0}\makebox(0,0)[lb]{\smash{$v_{i+1}$}}}%
    \put(0.53706762,0.34206757){\color[rgb]{0,0,0}\makebox(0,0)[lb]{\smash{$v^k_{n(k)}$}}}%
  \end{picture}%
\endgroup%

%% file: hex.pdf_tex
\begingroup%
  \makeatletter%
  \providecommand\color[2][]{%
    \errmessage{(Inkscape) Color is used for the text in Inkscape, but the package 'color.sty' is not loaded}%
    \renewcommand\color[2][]{}%
  }%
  \providecommand\transparent[1]{%
    \errmessage{(Inkscape) Transparency is used (non-zero) for the text in Inkscape, but the package 'transparent.sty' is not loaded}%
    \renewcommand\transparent[1]{}%
  }%
  \providecommand\rotatebox[2]{#2}%
  \ifx\svgwidth\undefined%
    \setlength{\unitlength}{377.37143555bp}%
    \ifx\svgscale\undefined%
      \relax%
    \else%
      \setlength{\unitlength}{\unitlength * \real{\svgscale}}%
    \fi%
  \else%
    \setlength{\unitlength}{\svgwidth}%
  \fi%
  \global\let\svgwidth\undefined%
  \global\let\svgscale\undefined%
  \makeatother%
  \begin{picture}(1,0.58873411)%
    \put(0,0){\includegraphics[width=\unitlength]{hex.pdf}}%
  \end{picture}%
\endgroup%